\documentclass[12pt,twoside]{amsart}
\usepackage{amsmath}
\usepackage{amsthm}
\usepackage{amsfonts}
\usepackage{amssymb}
\usepackage{latexsym}
\usepackage{mathrsfs}
\usepackage{amsmath}
\usepackage{amsthm}
\usepackage{amsfonts}
\usepackage{amssymb}
\usepackage{latexsym}
\usepackage{geometry}
\usepackage{dsfont}
\usepackage[dvips]{graphicx}
\usepackage{color}
\usepackage[all]{xy}

\date{}
\allowdisplaybreaks[4] \footskip=15pt
\renewcommand{\uppercasenonmath}[1]{}

\usepackage{graphicx,amssymb}
\usepackage[all]{xy}
\usepackage{amsmath}

\allowdisplaybreaks
\usepackage{amsthm}
\usepackage{color}

\theoremstyle{plain}
\newtheorem{theorem}{Theorem}[section]

\newtheorem{lemma}[theorem]{Lemma}

\newtheorem*{open question}{Open Question}

\theoremstyle{definition}
\newtheorem*{acknowledgement}{Acknowledgement}

\theoremstyle{remark}

\newcommand{\Add}{\mathrm{Add}}

\def\Hom{{\rm Hom}}

\begin{document}
	\begin{center}
		{\large  \bf A note on pure-semisimple rings}
		
		\vspace{0.5cm}     Xiaolei Zhang, Wei Qi
		
		\bigskip
		School of Mathematics and Statistics, Shandong University of Technology,\\
		Zibo 255049, P. R. China\\

	\end{center}
	
	\bigskip
	\centerline { \bf  Abstract}
	\bigskip
	\leftskip10truemm \rightskip10truemm \noindent
	
	In this note, we give several characterizations of  left pure-semisimple  in terms of the (pre)envelope, (pre)cover, direct limits, direct  sums, inverse limits and direct  products properties of pure-projective  modules or pure-injective modules.
	\vbox to 0.3cm{}\\
	{\it Key Words:}  pure-semisimple ring; pure-injective module; pure-projective module; cover and envelope.\\
	{\it 2010 Mathematics Subject Classification:} 16D60.
	
	\leftskip0truemm \rightskip0truemm
	\bigskip
	
	\section{Introduction}
	Throughout this article, $R$ is always  a ring with identity, and all modules are left modules.

	Let $M$ be an $R$-module and $\mathscr{C}$ be a  class  of $R$-modules. Recall from \cite{GT12} that  an $R$-homomorphism $f: M\rightarrow C$ with  $C\in \mathscr{C}$ is said to be a  \emph{$\mathscr{C}$-preenvelope}  provided that  for any $C'\in \mathscr{C}$, the natural homomorphism  $\Hom_R(C,C')\rightarrow \Hom_R(M,C')$ is an epimorphism. If, moreover, every endomorphism $h$ such that $f=hf$  is an automorphism, then $f: M\rightarrow C$ is said to be a
	\emph{$\mathscr{C}$-envelope}. Dually, one can define a \emph{$\mathscr{C}$-precover} and a \emph{$\mathscr{C}$-cover}.
	
	It is well-known that every $R$-module has a projective precover and an injective envelope for all rings $R$.
	In 1960, Bass \cite{B60} showed that a ring $R$ satisfies that every $R$-module has a projective cover, if and only if the class of projective $R$-module is closed under direct limits, if and only if $R$ is a left perfect ring, i.e., DCC holds for right principal ideals. Subsequently, Chase \cite{C60} obtained that the class of projective $R$-modules is closed under direct products if and only if $R$ is a left perfect and right coherent ring; and the class of injective $R$-modules is closed under direct sums (or direct limits) if and only if $R$ is a left Noetherian ring. In 1993, Asensio  and  Mart\'{i}ez \cite{AM93} showed that every $R$-modules has a projective (pre)envelope if and only if $R$ is a left perfect and right coherent ring.
	
	It is well-known that pure-injective modules  are $R$-modules $E$ satisfying that  $E\otimes_R-$ is an injective object in the functor category $(R\mbox{-mod}, Ab)$, i.e., the category of all additive functors from the category $R$\mbox{-mod} of all finitely presented $R$-modules to $Ab$ of all Abelian groups (see \cite[Theorem 12.1.6]{P09}). It also can be defined by the injectivity by using pure exact sequences. The notion of pure-projective modules can be defined dually.
	It is also well-known that every $R$-module has a pure-projective precover and a pure-injective envelope. A natural question is that how about the other related properties, such as the (pre)envelope, (pre)cover, direct limits, direct  sums, inverse limits and direct  products properties, of pure-projective modules or pure-injective modules. It is surprising all these properties characterize left pure-semisimple rings, i.e., rings over which all modules are pure-projective (pure-injective) (see Theorem \ref{main} for details).

	\section{main results}
	
	Recall that an $R$-module $C$ is said to be a \emph{pure-projective module} if any pure exact sequence $0\rightarrow A\rightarrow B\rightarrow C\rightarrow0$ ending at $C$ splits, or equivalently, if given any pure epimorphism $g: M\rightarrow N$ and any $R$-homomorphism $f:C\rightarrow N$, there is an $R$-homomorphism $h:C\rightarrow M$ such that the following diagram is commutative 	$$\xymatrix@R=25pt@C=40pt{
		& C\ar@{.>}[ld]_{h}\ar[d]^{f}\\
		M\ar@{->>}[r]^{\star}_{g}	&N\\}.$$
	The notion of a \emph{pure-injective module} can be defined dually.
	
	It is well known that an $R$-module is pure-projective if and only if it is a direct summand of a direct sum of finitely presented $R$-modules. So any direct sum or direct summand of a pure-projective module is pure-projective. It is well-known that every $R$-module has a pure-projective precover. However we can not find a suitable reference, so we exhibit it here.
	\begin{lemma}
		Let $R$ be a ring. Then every $R$-module has a 	pure-projective precover.
	\end{lemma}
	\begin{proof} Let $M$ be an $R$-module. It follows by \cite[Lemma 2.5]{GT12} that there is a family of direct system $\{F_i\mid i\in\Gamma\}$ of finitely presented $R$-modules such that ${\lim\limits_{\longrightarrow}}_{i\in\Gamma} F_i=M.$  So there is a pure exact sequence:
		$$0\rightarrow K\rightarrow \bigoplus\limits_{i\in\Gamma} F_i\xrightarrow{\pi} M\rightarrow 0.$$
		For any $R$-homomorphism $f:P\rightarrow M$ with $P$  pure-projective, there is an $R$-homomorphism $g: P\rightarrow\bigoplus\limits_{i\in\Gamma} F_i$ such that $f=\pi g$. Since $\bigoplus\limits_{i\in\Gamma} F_i$ is pure-projective, so $\pi: \bigoplus\limits_{i\in\Gamma} F_i\rightarrow M$ is a pure-projective precover of $M$.
	\end{proof}

	Dually, an $R$-module is pure-injective if and only if it is a direct summand of a direct product of dual of finitely presented $R$-modules. So any direct product or direct summand of a pure-injective module is pure-injective. Every $R$-module has a pure-injective envelope (see \cite[Proposition 6]{W69} or \cite[Theorem 6]{H03}).

	Recall that a ring $R$ is said to be a left \emph{pure-semisimple ring} if every $R$-module is pure-injective, or equivalently, every $R$-module is pure-projective. Some characterizations of pure-semisimple rings are given in \cite[Theorem 4.5.1,Theorem 4.5.7,Theorem 4.5.8, et. al]{P09}.

	The next main result is used to characterize pure-semisimple rings in terms of the (pre)envelope, (pre)cover, direct limits, direct  sums, inverse limits and direct  products properties of pure-projective  modules or pure-injective modules.

	\begin{theorem}\label{main} Let $R$ be a ring. Then the following statements are equivalent.
		\begin{enumerate}
			\item $R$ is a left pure-semisimple ring.
			\item Every $R$-module has a pure-projective envelope.
			\item Every $R$-module has a pure-projective preenvelope.
			\item   The class of pure-projective modules is closed under inverse limits.
			\item   The class of pure-projective modules is closed under direct products.
			\item   The class of pure-projective modules is closed under countably direct products.
			\item Every $R$-module has a pure-projective cover.
			\item The class of pure-projective modules is closed under direct limits.
			\item Every $R$-module has a pure-injective cover.
			\item Every $R$-module has a pure-injective precover.
			\item The class of pure-injective modules is closed under direct sums.
			\item The class of pure-injective modules is closed under countably direct sums.
			\item The class of pure-injective modules is closed under direct limits.
			\item The class of pure-injective modules is closed under inverse limits.
		\end{enumerate}
	\end{theorem}
	\begin{proof} $(1)\Rightarrow (2)\Rightarrow (3)$ and  $(1)\Rightarrow (4)\Rightarrow (5)\Rightarrow (6)$: Trivial.
		
		$(3)\Rightarrow (5)$:  Let $\{C_{i}\mid i\in\Lambda\}$ be a family of pure-projective modules and $\prod\limits_{i\in\Lambda}C_{i}\rightarrow C$ be a pure-projective preenvelope. Since each $C_i$	is pure-projective, there is a decomposition $$\prod\limits_{i\in\Lambda}C_{i}\rightarrow C\rightarrow C_{i}$$ for each $i\in\Lambda$. Hence the composition $$\prod\limits_{i\in\Lambda}C_{i}\rightarrow C\rightarrow \prod\limits_{i\in\Lambda}C_{i}$$ is the identity of $\prod\limits_{i\in\Lambda}C_{i}$. So $\prod\limits_{i\in\Lambda}C_{i}$ is a direct summand of $C$, and so is also a pure-projective $R$-module.

		$(6)\Rightarrow (1)$:	On contrary, suppose that $R$ is not pure-semisimple. It follows by \cite[Theorem 4.5.1]{P09} that there is such a strictly descending chain of pp-conditions with one free variable for left modules:
		$$\phi_1 > \phi_2 > \cdots> \phi_i >\phi_{i+1} > \cdots\ \ \ \ \ \  (\star)$$
		Choose a free realisation $(C_i,c_i)$ for each $\phi_i$, that is, $R$-module $C_i$ and element $c_i \in C_i$ such that the pp-type of $c_i$ in $C_i$ is generated by $\phi_i$ for each $i$ (see \cite[Proposition 1.2.14]{P09}). So each $C_i$ is a finitely presented, hence a pure-projective  $R$-module.
		
		Set $$C=\prod\limits_{i\geq 1} C_i.$$ We claim that $C$ is not pure-projective  $R$-module. Indeed, consider the element $c=(c_i)_{i\geq 1}$ in $C$.
		It follows  by \cite[Lemma 1.2.2]{P09} that the pp-type of $c$ in $C$ is generated by the set $\{ \phi_i\mid i\geq 1\}$. We claim that the pp-type of $c$ in $C$ is not finitely generated. Indeed, on contrary suppose this product $C$ were pure-projective. It follows by \cite[Theorem 1.3.22, Corollary 1.3.25]{P09} that  the pp-type of $c$ in $C$ would be finitely generated by some pp-type, which contradicts that  the chain $(\star)$ is strictly descending. Hence $C$ is not pure-projective. Consequently, we have a product of pure-projective modules which is not pure-projective.
		
		$(1)\Rightarrow (8)$: Trivial.
		
		$(8)\Rightarrow (7)$: It follows by \cite[Theorem 5.31]{GT12} and that every $R$-module has a pure-projective precover.
		
		$(7)\Rightarrow (1)$ Set $\mathscr{FP}$ the set of all finitely presented $R$-modules and let $$M=\bigoplus\limits_{F\in \mathscr{FP}}F.$$ Then $\Add(M)$ the set of all direct summand of copies of $M$ is exactly that of all pure-projective modules.
		It follows by \cite	[Corollary 2.3]{S23} that $\Add(M)$ is closed
		under direct limits if and only if $\Add(M)$ is a covering class. Note that every $R$-module is a direct limits of finitely presented modules (thus pure-projective modules). Hence $(1)$ holds.
		
		$(1)\Rightarrow (13)\Rightarrow (11)\Rightarrow (12)$ and $(1)\Rightarrow (9)\Rightarrow (10)$: Trivial.

		$(10)\Rightarrow (11)$: Let $\{N_{i}\mid i\in\Lambda\}$ be a family of pure-injective modules and $N\rightarrow \bigoplus\limits_{i\in\Lambda}N_{i}$ be a pure-injective precover. Since each $N_i$	is pure-injective, there is a decomposition $$N_{i}\rightarrow N\rightarrow  \bigoplus\limits_{i\in\Lambda}N_{i}$$ for each $i\in\Lambda$. Hence the composition $$\bigoplus\limits_{i\in\Lambda}N_{i}\rightarrow N\rightarrow \bigoplus\limits_{i\in\Lambda}N_{i}$$ is the identity of $\bigoplus\limits_{i\in\Lambda}N_{i}$. So $\bigoplus\limits_{i\in\Lambda}N_{i}$ is a direct summand of $N$, and so is also a pure-injective $R$-module.

		$(12)\Rightarrow (1)$  Let $N$ be an elementary cogenerator for the category of all $R$-modules (see \cite[Corollary 5.3.52]{P09}). Then  $N$ is $\Sigma$-pure-injective, i.e., $N^{(\aleph_0)}$ is also pure-injective. So every $R$-module purely embeds in a direct product of copies of $N$, and thus in the definable closure of the  $\Sigma$-pure-injective module $N$. It follows by \cite[Proposition 4.4.12]{P09} that every $R$-module is ($\Sigma$-)pure-injective and hence $R$ is pure-semisimple.
		
		$(1)\Rightarrow (14)$ Trivial.
		
		$(14)\Rightarrow (1)$ It follows by \cite[Theorem 2]{B13} that every $R$-module is an inverse limit of injective modules (and thus pure-injective modules). So every $R$-module is pure-injective, that is, $R$ is a left pure-semisimple ring.
	\end{proof}
	
	\begin{acknowledgement}\quad\\
		The first author was supported by  the National Natural Science Foundation of China (No.12061001), and the second author was supported by  the National Natural Science Foundation of China (No.12201361).
	\end{acknowledgement}

\end{document}